\RequirePackage{fix-cm}
\documentclass[11pt]{amsart} 
\usepackage{graphicx}
\usepackage{mathptmx}
\usepackage{latexsym,enumerate,hyperref,amsmath,amssymb,
amscd,enumerate,engord,amsfonts,graphics}
\usepackage[all]{xy}

\newtheorem{thm}{Theorem}
\newtheorem{lem}[thm]{Lemma}
\newtheorem{prop}[thm]{Proposition}
\newtheorem{cor}[thm]{Corollary}
\newtheorem{defin}[thm]{Definition}
\newtheorem{rem}[thm]{Remark}

\newcommand{\R}{{\mathbb{R}}}
\newcommand{\Z}{{\mathbb{Z}}}

\DeclareMathOperator{\TC}{{\sf TC}}

\DeclareMathOperator{\cat}{\mathrm{cat}}

\DeclareMathOperator{\Hom}{\mathrm{Hom}}
\DeclareMathOperator{\ev}{\mathrm{ev}}
\DeclareMathOperator{\zcl}{\mathrm{zcl}}

\begin{document}

\title{Topological complexity of the Klein bottle}
\author[Daniel C. Cohen]{Daniel C. Cohen\textsuperscript{\dag}}
\address{Department of Mathematics, Louisiana State University, Baton Rouge, LA, USA 70808}
\email{\href{mailto:cohen@math.lsu.edu}{cohen@math.lsu.edu}}
\urladdr{\href{http://www.math.lsu.edu/~cohen/}{www.math.lsu.edu/\char'176cohen}}
\thanks{\textsuperscript{\dag}~D.C.\,partially supported by NSF 1105439.}

\author[Lucile Vandembroucq]{Lucile Vandembroucq\textsuperscript{\ddag}}
\address{Centro de Matem\'atica, Universidade do Minho, Campus de Gualtar, 4710-057 Braga, Portugal}
\email{\href{mailto:lucile@math.uminho.pt}{lucile@math.uminho.pt}}\thanks{\textsuperscript{\ddag}~L.V.\,partially supported by Portuguese Funds through FCT -- Funda\c c\~ao para a Ci\^encia e a Tecnologia, within the Project UID/MAT/00013/2013.}

\begin{abstract}
We show that the normalized topological complexity of the Klein bottle is equal to $4$. For any non-orientable surface $N_g$ of genus $g\geq 2$, we also show that $\TC(N_g)=4$. This completes recent work of Dranishnikov on the topological complexity of non-orientable surfaces.
\end{abstract}

\maketitle

\section{Introduction}
\label{intro}
The topological complexity of a space $X$, $\TC(X)$, is a homotopy invariant introduced by M. Farber in \cite{Far} in order to give a topological measure of the complexity of the motion planning problem in robotics. We consider here the normalized version of $\TC$, that is, $\TC(X)$ is the least integer $n$ such that there exists a cover of $X\times X$ by $n+1$ open sets, on each of which the fibration
\[
\ev_{0,1}\colon X^I\to X\times X, \quad \gamma \mapsto (\gamma(0),\gamma(1)),
\]
admits a continuous local section.
Roughly speaking, if $X$ is the space of all possible states of a mechanical system, then we need at least $\TC(X)+1$ rules to determine a complete algorithm which dictates how the system will move from any initial state to any final state. Refer to \cite{FarInv} for further discussion, and as a general reference.

We have, for a CW complex $X$, the following estimates of $\TC(X)$ (see \cite{Far}):
\begin{equation*}
\max\{\cat(X), \zcl_{\Bbbk}(X)\} \leq \TC(X) \leq 2\cat(X)\leq 2\dim(X).
\end{equation*}
Here, $\cat(X)$ is the normalized Lusternik-Schnirelmann category of $X$ and $\zcl_{\Bbbk}(X)$ denotes the \textit{zero-divisors cup-length} of the cohomology of $X$ with coefficients in a field $\Bbbk$. More precisely, $\zcl_{\Bbbk}(X)$ is the nilpotency of the kernel of the cup product $H^*(X;\Bbbk)\otimes H^*(X;\Bbbk)\to H^*(X;\Bbbk)$, the smallest nonnegative integer $n$ such that any $(n+1)$-fold cup product in this kernel is trivial. 

As is well-known, determining the topological complexity of non-orientable surfaces $N_g$ has turned out to be a difficult task. 
For any $g\ge 1$, the aforementioned dimensional upper bound and the zero-divisors cup-length with $\Z_2$ coefficients yield
\[
3\leq \TC(N_g)\leq 4.
\]
The actual value of $\TC$ for $N_1=\R P^2$, equal to $3$, has been first determined through the Farber-Tabachnikov-Yuzvinsky theorem of \cite{FTY} relating  $\TC(\R P^n)$ to the immersion dimension of $\R P^ n$, 
and later recovered by Costa and Farber \cite{FarberCosta} through a computation of zero-divisors cup-length with local coefficients. Recently A. Dranishnikov established in \cite{Dranishnikov} that $\TC(N_g)=4$ for $g\geq 5$. He also showed in \cite{Dranishnikov2} that $\TC(N_4)=4$ and that his methods do not extend to the lower genus cases $g\in \{2,3\}$. Since then, only the case of $N_g$ with $g\in \{2,3\}$ and in particular the case of the Klein bottle $K=N_2$ were still open. Here, we prove:
\begin{thm}\label{TC(K)} The normalized topological complexity of $K$ is equal to $4$.
\end{thm}
\begin{thm} \label{TC(Ng)} For any $g\geq 2$, $\TC(N_g)=4$.
\end{thm}

As in \cite{FarberCosta}, our results will follow from computations of zero-divisors cup-lengths with local coefficients. 
For a discrete group $\pi$, we denote by $\Z[\pi]$ the associated integral group ring and by $I(\pi)=\ker(\varepsilon \colon \Z[\pi]\to \Z)$ the augmentation ideal. We will see that, when $\pi=\pi_1(N_g)$ and $g\geq 2$, the fourth power of a particular zero-divisor $\mathfrak{v}\in H^1(N_g\times N_g;I(\pi))$ introduced in \cite{FarberCosta} is not trivial, that is
\[
{\mathfrak{v}}^{4}\neq 0\  \text{in}\  H^{4}(N_g\times N_g;I(\pi)^{\otimes 4}),
\]
from which we obtain that $\TC(N_g)\geq 4$. Our computations will be based on an explicit calculation, in the case of the Klein bottle, of a cocycle representing $ {\mathfrak{v}}^{4}$ and the evaluation of this cocycle on a cycle which is not homologically trivial. 

The techniques utilized here inform on the homotopy cofibre $C_\Delta$ of the diagonal map $\Delta\colon K \to K \times K$ of the Klein bottle. In \cite{Dranishnikov}, Dranishnikov shows that the (normalized) Lusternik-Schnirelmann category of $C_\Delta$ is equal to $3$. Our methods may be used to recover this result. Discussion of this application will appear elsewhere~\cite{CV}.

Our results also complete recent work of J. Gonz\'alez, B. Guti\'errez, D. Guti\'errez and A. Lara \cite{GGGL} on the higher topological complexity of the non-orientable surfaces. Recall that the $s$-topological complexity of a space $X$, $\TC_s(X)$, 
introduced by Y. Rudyak in \cite{Rudyak}, satisfies $\TC_2=\TC$. 
Using the (higher) zero-divisors cup-length with coefficients in $\Z_2$, the authors of \cite{GGGL} compute $\TC_s(N_g)$ for any $s\geq 3$ and $g\geq 1$.

Lastly, in the case $g=2$, D.\,Davis independently develops an alternate approach to the obstruction ${\mathfrak{v}}^{4}$  based on a $\triangle$-complex structure of the Klein~bottle in \cite{Davis}. 

\section{Preliminaries} \label{sec:prelim}
\subsection{Some notations} \label{sec:notation}
We recall here some standard notations and fix some conventions we will use throughout the paper. For a discrete group $\pi$ (with unit $1$), we denote by $\bar{a}$ the inverse of an element $a\in \pi$ and by $\Z[\pi]$ the associated integral group ring. All our $\Z[\pi]$-modules will be left modules. If $M$ is a $\Z[\pi]$-module, we denote by $M_G$ the coinvariants of $M$, 
\[
M_G=M/{\rm span}_{\Z}\{m-a\cdot m \,|\, m\in M, a\in \pi\}.
\]
For (left) $\Z[\pi]$-modules $I$ and $J$, the tensor product $I\otimes J=I\otimes_{\Z}J$ is a left $\Z[\pi]$-module via the diagonal action of $\pi$, and we denote by  $I\otimes_{\pi}J$ the coinvariants of $I\otimes J$.

\subsection{Topological complexity and canonical TC class} \label{sec:tc class}
As mentioned in the introduction, we consider the normalized version of Farber's topological complexity:

\begin{defin} The topological complexity of a topological space $X$, $\TC(X)$, is the least integer $n$ such that there exists a cover of $X\times X$ by $n+1$ open sets on each of which the fibration
\[
\ev_{0,1}\colon X^I\to X\times X, \quad \gamma \mapsto (\gamma(0),\gamma(1)),
\]
admits a continuous local section.
\end{defin}

As in \cite{FarInv}, we will consider the cohomological lower bound of $\TC$ given by the \textit{zero-divisor cup-length} in cohomology with local coefficients. Recall that a local system of coefficients on a space $Y$ is a $\Z[\pi_1(Y)]$-module.

Let $A$ be a local coefficient system on $X\times X$ and let $A|X$ be the local system on $X$ induced by the diagonal map $\Delta\colon X\to X\times X$. A cohomology class $\upsilon\in H^*(X\times X;A)$ is called a \textit{zero-divisor} if 
\[
\Delta^*(\upsilon)=0\ \text{in}\  H^*(X;A|X).
\]
By \cite[Corollary 4.40]{FarInv}, if the cup product of $n$ zero-divisors $u_i\in H^*(X\times X;A_i)$ is nonzero 
in $H^*(X\times X; A_1\otimes\cdots\otimes A_n)$, then $\TC(X)\geq n$.

In particular, if we take coefficients in a field $\Bbbk$ (with the group ring $\Z[\pi_1(X\times X)]$ acting trivially), the zero-divisors can be identified with the elements of the kernel of the cup product $\cup\colon H^*(X;{\Bbbk})\otimes H^*(X;{\Bbbk})\to H^*(X;{\Bbbk})$ and \cite[Corollary 4.40]{FarInv} specializes to \cite[Theorem 7]{Far} on the zero-divisors cup-length of $H^*(X;{\Bbbk})$.

In \cite{FarberCosta}, Costa and Farber associate to a space $X$ a canonical zero-divisor, which we call the \textit{canonical $\TC$ class} and describe next. Let $\pi=\pi_1(X)$ be the fundamental group of $X$.  Let $I(\pi)=\ker(\varepsilon\colon\Z[\pi]\to \Z)$ be the augmentation ideal. Recall that  $\Z[\pi]$ and $I(\pi)$ are both (left) $\Z[\pi\times \pi]$-modules through the action given by:
\[
(a, b) \cdot \sum n_ia_i= \sum n_i (aa_i\bar{b}).
\]
Here $n_i\in \Z$ and $a,b,a_i\in \pi$.
We denote by ${\mathfrak{v}}={\mathfrak{v}}_X\in H^ 1(X\times X;I(\pi))$ the cohomology class induced by the crossed homomorphism (see \cite{FarberCosta} or the next section)
\[
\pi\times \pi \to I(\pi), \,\,(a, b)\mapsto a\bar{b}-1.
\]
The class ${\mathfrak{v}}$ is a zero-divisor and from \cite{FarberCosta} we have:
\begin{thm}(\cite[Theorem 7]{FarberCosta})\label{FaberCosta}  Suppose that $X$ is a CW-complex of dimension $n\geq 2$. Then $\TC(X)=2n$ if and only if the $2n$-th power of ${\mathfrak{v}}$ does not vanish:
\[
\TC(X)= 2n \Longleftrightarrow  {\mathfrak{v}}^{2n}\neq 0\ \text{in}\   H^{2n}(X\times X;I(\pi)^{\otimes 2n}).
\]
\end{thm}
Here $I(\pi)^{\otimes 2n}=I(\pi)\otimes_\Z I(\pi) \otimes_\Z \cdots \otimes_\Z I(\pi)$ is the tensor product of $2n$ copies of $I(\pi)$, with the diagonal action of $\pi\times \pi$.

If $X=N_g$ is a non-orientable surface, it is well-known (and established through an easy calculation) that the zero-divisors cup-length of $H^ *(N_g;\Z_2)$ is equal to $3$, which implies that $\TC(N_g)\geq 3$. Therefore, $3\leq \TC(N_g)\leq 4$ and, if $\pi=\pi_1(N_g)$, Theorem \ref{FaberCosta} specializes to 
\begin{cor}\label{FaberCostaNg} $\TC(N_g)= 4 \Longleftrightarrow {\mathfrak{v}}^{4}\neq 0\ \text{in}\  H^{4}(N_g\times N_g;I(\pi)^{\otimes 4})$.
\end{cor}

We next use the bar resolution to give explicit expressions of ${\mathfrak{v}}_{N_g}$ and its powers.

\subsection{Bar resolution and canonical cocyle} \label{sec:bar}

Let $\pi$ be a discrete group and $X$ a $K(\pi,1)$-space. In this section we exhibit, for $n\geq 1$, a canonical cocycle representing the cohomology class ${\mathfrak v}^ n\in H^n(X\times X; I(\pi)^{\otimes n})$.
  
Recall (see, for instance, Brown \cite{Brown}) the bar resolution $B_{*}(\pi)$ of $\Z$ as a trivial $\Z[\pi]$-module:
\[
\cdots \longrightarrow B_n(\pi)\xrightarrow{\ \partial_n\ } \cdots \longrightarrow B_1(\pi) \xrightarrow{\ \partial_1\ } B_0(\pi)=\Z[\pi] \xrightarrow{\ \varepsilon\ } \Z \longrightarrow 0.
\]
Here $B_n(\pi)$ is the free $\Z[\pi]$-module with basis $\{[a_1|\cdots |a_n], (a_1,\cdots ,a_n)\in \pi^n\}$ and $\partial_n$ is the $\Z[\pi]$ morphism given by 
\begin{multline*}
\partial_n([a_1|\cdots |a_n])=a_1\cdot [a_2|\cdots |a_n]+\sum\limits_{i=1}^{n-1}(-1)^i[a_1|\cdots |a_{i-1}|a_ia_{i+1}|a_{i+2}|\cdots |a_n]\\ + (-1)^n[a_1|\cdots |a_{n-1}].
\end{multline*}
In particular, $\partial_1[a]=a[\,]-[\,]=a-1$. As a $\Z[\pi]$-chain complex, $B_*(\pi)$ is equivalent to the singular chain complex of the universal cover $\widetilde{X}$ of $X$.

If $A$ is a $\Z[\pi]$-module then $B_{*}(\pi)\otimes_{\pi}A$ (with differential $\partial \otimes\, \mathrm{id}$) is a $\Z$-chain complex that we denote by $B_{*}(\pi;A)$ and call the bar resolution with coefficients in $A$. We also consider the $\Z$-cochain complex of cochains with coefficients, $C^{*}(\pi;A):=(\Hom_{\Z [\pi]}(B_{*}(\pi),A),\delta)$, where $\delta$ is given by $\delta_n \alpha= (-1)^{n+1}\alpha \partial_{n+1}$ and $A$ is viewed as a chain complex concentrated in degree $0$. A cochain $\alpha$ of degree $n$ is completely determined by a $\Z[\pi]$-morphism $B_n(\pi)\to A$ and will be so indicated. 
In this way, the (co)homology groups of a $K(\pi,1)$-space $X$ with coefficients in $A$ are given by 
\[
H_*(X;A)=H (B_*(\pi;A)) \quad \text{and}\quad H^*(X;A)=H(C^*(\pi;A)).
\]
Recall also that these (co)homology groups may be computed using $\Z$-(co)chain-complexes defined analogously to $B_*(\pi;A)$ and $C^*(\pi;A)$ by replacing $B_*(\pi)$ with any other free $\Z[\pi]$-resolution of $\Z$. In particular we will make use of finite free resolutions in Sections \ref{fundcycle} and \ref{SectionNg}.

The Alexander-Whitney diagonal is the $\Z[\pi]$-chain map given by
\[
\begin{array}{rcl}
\Delta\colon B_*(\pi) & \to & B_*(\pi)\otimes B_*(\pi)\\
{[}a_1|\cdots |a_n] &\mapsto & \sum\limits_{i=0}^{n} [a_1|\cdots |a_i]\otimes a_1\cdots a_i[a_{i+1}|\cdots |a_n]
\end{array}
\]
where the action of $\Z[\pi]$ on $B_*(\pi)\otimes B_*(\pi)$ is induced by the diagonal action of $\pi$.

If $\alpha\colon B_i(\pi)\to A_1$ and $\beta\colon B_{n-i}(\pi)\to A_2$ are cochains of degrees $i$ and $n-i$ respectively, their cup product $\alpha \cup \beta$ is the cochain of degree $n$ given by 
\[
\alpha \cup \beta\colon B_n(\pi) \xrightarrow{\ \Delta\ } (B_*(\pi)\otimes B_*(\pi))_n \xrightarrow{\ (-1)^{i(n-i)}\alpha \otimes \beta\ } A_1\otimes A_2.
\]
We have $\delta (\alpha \cup \beta)=\delta \alpha \cup \beta + (-1)^{\deg(\alpha)}\alpha \cup \delta \beta$.

Since $B_*(\pi\times \pi)$ is, as a $\Z[\pi\times \pi]$-chain complex, equivalent to the singular chain complex of $\widetilde{X}\times \widetilde{X}$, using the proof of \cite[Lemma 5]{FarberCosta}, it is readily checked that the canonical $\TC$ cohomology class ${\mathfrak v}\in H^ 1(X\times X;I(\pi))$ is the class of the canonical cocycle of degree $1$, $\nu\colon B_1(\pi\times \pi)\to I(\pi)$, given by
\[
\nu( [(a, b)])= a\bar{b}-1
\]
for $[(a,b)]\in B_1(\pi\times \pi)$. Using the Alexander-Whitney diagonal we then obtain the following explicit expression of the $n$-th power of ${\mathfrak v}\in H^ 1(X\times X;I(\pi))$:

\begin{lem}\label{explicitpower} The $n$-th power of the canonical $\TC$ cohomology class $\mathfrak{v}$ is the class of the cocyle $\nu^n$ of degree $n$ given by
{
\setlength\arraycolsep{2pt}
\begin{eqnarray*}
\nu^n\colon B_n(\pi\times \pi) &\to & I(\pi)^{\otimes n}\\
\left[(a_1,b_1)|\cdots |(a_n, b_n)\right] & \mapsto & \xi\cdot \bigl(u_1-1\bigr)\otimes
\cdots \otimes (a_1\cdots a_{n-1})\bigl(u_n-1\bigr)(\bar{b}_{n-1}\cdots  \bar{b}_{1}),\nonumber
\end{eqnarray*}
}
where $\xi=(-1)^{n(n-1)/2}$ and $u_i=a_i\bar{b}_i$ for each $i$, $1\le i\le n$.
\end{lem}

We will also use the Eilenberg-Zilber chain equivalence 
\begin{equation} \label{eq:EZ}
EZ\colon B_*(\pi)\otimes B_*(\pi) \longrightarrow B_*(\pi \times \pi),
\end{equation}
which is the $\Z[\pi\times \pi]\cong\Z[\pi]\otimes \Z[\pi]$ morphism given by
\[
\begin{array}{rcl}
EZ_n\colon \bigoplus\limits_{i=0}^{n}B_i(\pi)\otimes B_{n-i}(\pi)& \to & B_n(\pi\times \pi) \\
{[}a_1|\cdots |a_i]\otimes [b_{i+1}|\cdots |b_n] &\mapsto & \sum\limits_{\sigma \in {\mathcal S}_{i,n-i}}\mathrm{sgn}(\sigma)[c_{\sigma^{-1}(1)}| \cdots | c_{\sigma^{-1}(n)}]
\end{array}
\]
where ${\mathcal S}_{i,n-i}$ denotes the set of $(i,n-i)$ shuffles, $\mathrm{sgn}(\sigma)$ is the signature of the shuffle $\sigma$, and
\[
c_{k}=\begin{cases}
(a_{k}, 1) & \mbox{ if } 1\leq k \leq i, \\
(1, b_{k}) & \mbox{ if } i+1\leq k \leq n.
\end{cases}
\]

\subsection{Corollary \ref{FaberCostaNg} revisited}
Let $X=N_g$ and $\pi=\pi_1(N_g)$. In order to see that ${\mathfrak{v}}^{4}$ is not zero, it is sufficient to see that the evaluation of the cocycle
\[
\nu^4\colon B_4(\pi\times \pi) \longrightarrow I(\pi)^{\otimes 4}
\]
on a homologically nontrivial cycle is nonzero. A natural choice for such a homologically nontrivial cycle is a cycle corresponding to the twisted fundamental class of $N_g\times N_g$. Recall that the twisted fundamental class of $N_g\times N_g$, denoted by $[N_g\times N_g]$, is a generator of $H_4(N_g\times N_g, \widetilde{\Z})\cong \Z$ where $\widetilde{\Z}$ denotes the orientation module of $N_g\times N_g$, that is, $\widetilde{\Z}$ is the free abelian group $\Z$ given with the structure of $\Z[\pi\times \pi]$-module induced by the first Stiefel-Whitney class $w_1(N_g\times N_g)\in H^1(N_g\times N_g;\Z_2)$. 

Letting  $t$ denote a generator of $\widetilde{\Z}$, the class $[N_g\times N_g]$ can be represented by a cycle
\[
\Omega_{\widetilde{\Z}}=\Omega \otimes_{\pi\times \pi} t \in B_4(\pi\times \pi;\widetilde{\Z}),
\]
where $\Omega \in B_4(\pi\times \pi)$.

The evaluation of $\nu^4$ on $\Omega_{\widetilde{\Z}}$, which we denote by $\nu^4(\Omega_{\widetilde{\Z}})$, is the image of $\Omega_{\widetilde{\Z}}$ under the following morphism:
\[
B_{4}(\pi\times \pi)\otimes_{\pi\times \pi}\widetilde{\Z} \xrightarrow{\ \nu^4\otimes\, \mathrm{id}\ } I(\pi)^4\otimes_{\pi\times \pi}\widetilde{\Z},
\]
that is, $\nu^4(\Omega_{\widetilde{\Z}})=\nu^4(\Omega)\otimes_{\pi\times \pi}t$. Notice that   $\nu^4(\Omega_{\widetilde{\Z}})$ is exactly the cap-product of ${\mathfrak v}^4$ with $[N_g\times N_g]$:
\[
{\mathfrak v}^4 \cap [N_g\times N_g]\in H_0(N_g\times N_g; I(\pi)^{\otimes 4}\otimes \widetilde{\Z})=I(\pi)^{\otimes 4}\otimes_{\pi\times \pi}\widetilde{\Z}.
\]
By Poincar\'e duality we have an isomorphism
\[
\cap [N_g \times N_g]\colon H^4(N_g\times N_g; I(\pi)^{\otimes 4}) \xrightarrow{\ \cong\ } H_0(N_g\times N_g; I(\pi)^{\otimes 4}\otimes \widetilde{\Z}).
\]
Consequently, Corollary \ref{FaberCostaNg} can be continued as follows:

\begin{prop} \label{FaberCostaNgRevisited}
 $\TC(N_g)= 4 \Longleftrightarrow \nu^4(\Omega_{\widetilde{\Z}})\neq 0\ \text{in}\ I(\pi)^{\otimes 4}\otimes_{\pi\times \pi} \widetilde{\Z}$.
\end{prop}

Actually we will see that, for $g\geq 2$, $\nu^4({\Omega_{\Z_2}})\neq 0$ where $\Z_2$ is equipped with the trivial $\pi\times \pi$ action and ${\Omega}_{\Z_2}\in B_4(\pi\times \pi)\otimes_{\pi\times \pi}\Z_2=B_4(\pi\times \pi;\Z_2)$ is a cycle representing the generator of $H_4(N_g\times N_g;\Z_2)$, that is, the $\Z_2$ fundamental class of $N_g\times N_g$. The following commutative diagram, in which the vertical maps are induced by the obvious projection $\widetilde{\Z}\to \Z_2$,
\[
\xymatrix{
B_{4}(\pi\times \pi;\widetilde{\Z})=B_{4}(\pi\times \pi)\otimes_{\pi\times \pi}\widetilde{\Z}\ar[d] \ar[rr]^-{\nu^4\otimes\, \mathrm{id}}&& I(\pi)^{\otimes 4}\otimes_{\pi\times \pi}\widetilde{\Z}\ar[d]\\
B_{4}(\pi\times \pi;\Z_2)=B_{4}(\pi\times \pi)\otimes_{\pi\times \pi}{\Z_2} \ar[rr]^-{\nu^4\otimes\, \mathrm{id}}&& I(\pi)^{\otimes 4}\otimes_{\pi\times \pi}{\Z_2}
}
\]
ensures that $\nu^4(\Omega_{\widetilde{\Z}})\neq 0$ as soon as $\nu^4({\Omega_{\Z_2}})\neq 0$, since $\Omega_{\Z_2}$ is the image of $\Omega_{\widetilde{\Z}}$ under the left-hand vertical map. So, Theorems \ref{TC(K)} and \ref{TC(Ng)} will follow from the following statement, together with the results of Sections \ref{SectionKlein} and \ref{SectionNg}, on the Klein bottle and higher genus surfaces, respectively.

\begin{prop}\label{Z2coeff} If $\nu^4({\Omega_{\Z_2}})\neq 0\ \text{in}\ I(\pi)^{\otimes 4}\otimes_{\pi\times \pi}{\Z_2}$, then $\TC(N_g)=4$.
\end{prop}

\section{Klein bottle}\label{SectionKlein}
We now consider the special case of the Klein bottle:
\begin{equation} \label{eq:Kpres}
K=N_2, \text{ a }K(G,1)\text{-space,} \quad \mbox{where} \quad G=\langle x,y \,|\, yxy=x\rangle.
\end{equation} 
Note that $G=\Z\rtimes\Z$, where the action of $\Z=\langle x\rangle$ on $\Z=\langle y\rangle$ is given by $x^{-1}yx=y^{-1}$.
Theorem \ref{TC(K)} will follow from Proposition \ref{Z2coeff} together with the following result we will establish in this section:
\begin{prop} \label{propKlein} $\nu^4({\Omega_{\Z_2}})\neq 0\ \text{in}\ I(G)^{\otimes 4}\otimes_{G\times G}{\Z_2}$.
\end{prop}

Since the action of $G\times G$ on $\Z_2$ is trivial, $I(G)\otimes \Z_2$ can be identified with the $\Z[G\times G]$-module $I(G;\Z_2)=\ker(\varepsilon\colon \Z_2[G]\to \Z_2)$, and we have
$$ I(G)^{\otimes 4}\otimes_{G\times G}{\Z_2}=({I}(G;\Z_2)^{\otimes 4})_{G\times G}.$$
We will first determine an explicit expression of a cycle ${\Omega}_{\Z_2} \in {B}_4(G\times G;\Z_2)$ representing the generator of $H_4(K\times K;\Z_2)$, then calculate $\nu^4({\Omega}_{\Z_2})$ and prove that $\nu^4({\Omega_{\Z_2}})$ is nonzero in $({I}(G;\Z_2)^{\otimes 4})_{G\times G}$.

\subsection{Explicit fundamental cycle}\label{fundcycle}
One can construct a finite, free resolution of $\Z$ as $\Z[G]$-module by applying the Fox free differential calculus to the presentation 
of the Klein bottle group $G$ given above. 
References for the Fox calculus include Fox's original papers \cite{Fox} and Birman \cite[\S 3.1]{Birman}, and additional exposition may be found in \cite{CohenSuciu}. 
Explicitly, using the presentation \eqref{eq:Kpres} and the fact that $G$ has cohomological dimension $2$, a finite, free $\Z[G]$-resolution of $\Z$ is given by:
\begin{equation} \label{eq:res}
\xymatrix{
P_2\ar[r]^{\partial} & P_1 \ar[r]^{\partial} & P_0 \ar[r]^{\varepsilon}&\Z
}
\end{equation}
where $P_0,P_1,P_2$ are free $\Z[G]$-modules with respective bases $\{e^0\}$, $\{e^1_1,e^1_2\}$, $\{e^2\}$ and the differential is given by
\[
\partial(e^1_1)=(x-1)e^0,\quad\partial(e^1_2)=(y-1)e^0,\quad \partial(e^2)= (y-1)e^1_1+(1+yx)e^1_2.
\]

A chain map, which is necessarily a chain equivalence (see, e.g., \cite[\S I.7]{Brown}), from the resolution $(P_{\bullet},\partial)$ to the bar resolution $B_{*}(G)$ is given by 
\[
e^0\mapsto 1,\quad  e^1_1\mapsto [x],\quad  e^1_2\mapsto [y], \quad e^2 \mapsto [y|x]+[yx|y].
\]
By tensoring with $\Z_2$ over $G$, $e^2$ gives a cycle whose homology class is the generator of $H_ 2(K;\Z_2)$. Therefore,
\[
\kappa=[y|x]+[yx|y]\in B_2(G)
\]
induces a representative, ${\kappa}_{\scriptstyle{\Z_2}}$, of this $\Z_2$ fundamental class in $B_2(G;\Z_2)$. 

\begin{rem} The orientation module of the Klein bottle $K$, $\widetilde {\Z}=\Z t$, has $\Z[G]$-module structure generated by:
\[
x^p \cdot t=(-1)^pt \qquad y^p \cdot t=t \qquad (p\in \Z).
\]
The elements $e_2\otimes_{G} t\in P_{2}\otimes_G \widetilde{\Z}$ and $\kappa {\otimes}_{G}t=\kappa_{\widetilde{\Z}}\in B_2(G,\widetilde{\Z})$ are cycles representing the twisted fundamental class $[K]\in H_2(K;\widetilde{\Z})=\Z$ of $K$.
\end{rem}

A cycle, $\Omega_{\Z_2} \in B_4(G\times G;\Z_2)$ representing the generator of $H_4(K\times K; \Z_2)$ is then induced by $\Omega=EZ(\kappa\otimes \kappa)\in B_4(G\times G)$ where $EZ$ is the Eilenberg-Zilber map \eqref{eq:EZ}, which is a chain equivalence. Explicitly:
\begin{equation} \label{eq:Omega}
\arraycolsep=1.4pt\def\arraystretch{1.2}
\Omega=\left\{\begin{array}{c}
+[y_1|x_1|y_2|x_2]-[y_1|y_2|x_1|x_2]+[y_1|y_2|x_2|x_1]\\
+[y_2|y_1|x_1|x_2]-[y_2|y_1|x_2|x_1]+[y_2|x_2|y_1|x_1]\\
+[y_1|x_1|y_2x_2|y_2]-[y_1|y_2x_2|x_1|y_2]+[y_1|y_2x_2|y_2|x_1]\\
+[y_2x_2|y_1|x_1|y_2]-[y_2x_2|y_1|y_2|x_1]+[y_2x_2|y_2|y_1|x_1]\\
+[y_1x_1|y_1|y_2|x_2]-[y_1x_1|y_2|y_1|x_2]+[y_1x_1|y_2|x_2|y_1]\\
+[y_2|y_1x_1|y_1|x_2]-[y_2|y_1x_1|x_2|y_1]+[y_2|x_2|y_1x_1|y_1]\\
+[y_1x_1|y_1|y_2x_2|y_2]-[y_1x_1|y_2x_2|y_1|y_2]+[y_1x_1|y_2x_2|y_2|y_1]\\
+[y_2x_2|y_1x_1|y_1|y_2]-[y_2x_2|y_1x_1|y_2|y_1]+[y_2x_2|y_2|y_1x_1|y_1]
\end{array}\right.
\end{equation}
where $x_1=(x,1)$, $x_2=(1,x)$, $y_1=(y,1)$, $y_2=(1,y)$. Notice that, with this notation
$u_1v_2=v_2u_1$ for $u,v\in\{x,y\}$.

\begin{rem} The expression \eqref{eq:Omega} also provides a cycle $\Omega_{\widetilde{\Z}}\in B_4(G\times G; \widetilde{\Z})$ which represents the twisted fundamental class of $K\times K$. Here, the structure of the orientation module of $K\times K$ is given by 
\[
x_1^p \cdot t=x_2^p\cdot t=(-1)^pt, \qquad y_1^p \cdot t=y_2^p \cdot t=t, \qquad (p\in \Z).
\]
Alternatively, one can
use the product structure of $G\times G$ (or express $G\times G$ as an iterated semidirect product of rank $1$ free groups and use  \cite{CohenSuciu}) to 
exhibit a finite, free $\Z[G\times G]$-resolution $F_*$ of $\Z$. For instance, one can take $F_*=P_*\otimes P_*$, where $P_*$ is the $\Z[G]$-resolution \eqref{eq:res}, see \cite[Ch.\,V]{Brown}. Then, constructing a chain map $F_* \to B_*(G\times G)$, 
the image of the unique degree $4$ generator of $F_*$ yields an explicit expression of a representative of the twisted fundamental class of $K\times K$ as above. 
\end{rem}

\subsection{Explicit expression of $\nu^ 4(\Omega)$.} 

By applying the cocyle given in Lemma \ref{explicitpower} to $\Omega$, we obtain the following expression of $\nu^4(\Omega)\in I(G)^{\otimes 4}$. Here we only use the relations coming from the relation $yxy=x$ (e.g., $yx=x\bar{y}$,  $y\bar{x}=\bar{x}\bar{y}$, etc.).  For convenience in future calculations, we label each term of this expression by $(T_i)$, as indicated in the righthand column.
\[
\arraycolsep=5pt\def\arraystretch{1.1}
\nu^4(\Omega)=
{\scriptstyle{
\left\{\begin{array}{lr}
+(y-1)\otimes (yx-y) \otimes (y^2x-yx) \otimes (1-y^2x) &  (T_1)\\
-(y-1)\otimes (1-y) \otimes (y^2x-1) \otimes (1-y^2x) & (T_2)\\
+(y-1)\otimes (1-y) \otimes (y^2\bar{x}-1) \otimes (1-y^2\bar{x}) & (T_3)\\
+(\bar{y}-1)\otimes (1-\bar{y}) \otimes (y^2x-1) \otimes (1-y^2x) & (T_4)\\
-(\bar{y}-1)\otimes (1-\bar{y}) \otimes (y^2\bar{x}-1) \otimes (1-y^2\bar{x}) & (T_5)\\
+(\bar{y}-1)\otimes (y\bar{x}-\bar{y}) \otimes (y^2\bar{x}-y\bar{x}) \otimes (1-y^2\bar{x}) & (T_6)\\
+(y-1)\otimes (yx-y) \otimes (1-yx) \otimes (y-1) & (T_7)\\
-(y-1)\otimes (y^2\bar{x}-y) \otimes (1-y^2\bar{x}) \otimes (y-1) & (T_8)\\
+(y-1)\otimes (y^2\bar{x}-y) \otimes (y\bar{x}-y^2\bar{x}) \otimes (y-y\bar{x}) & (T_9)\\
+(y\bar{x}-1)\otimes (y^2\bar{x}-y\bar{x}) \otimes (1-y^2\bar{x}) \otimes (y-1) & (T_{10})\\
-(y\bar{x}-1)\otimes (y^2\bar{x}-y\bar{x}) \otimes (y\bar{x}-y^2\bar{x}) \otimes (y-y\bar{x})  & (T_{11})\\
+(y\bar{x}-1)\otimes (\bar{x}-y\bar{x}) \otimes (y\bar{x}-\bar{x}) \otimes (y-y\bar{x})  & (T_{12})\\
+(yx-1)\otimes (x-yx) \otimes (yx-x) \otimes (\bar{y}-yx)  & (T_{13})\\
-(yx-1)\otimes (y^2x-yx) \otimes (yx-y^2x) \otimes (\bar{y}-yx)  & (T_{14})\\
+(yx-1)\otimes (y^2x-yx) \otimes (1-y^2x) \otimes (\bar{y}-1)  & (T_{15})\\
+(\bar{y}-1)\otimes (y^2x-\bar{y}) \otimes (yx-y^2x) \otimes (\bar{y}-yx)  & (T_{16})\\
-(\bar{y}-1)\otimes (y^2x-\bar{y}) \otimes (1-y^2x) \otimes (\bar{y}-1)  & (T_{17})\\
+(\bar{y}-1)\otimes (y\bar{x}-\bar{y}) \otimes (1-y\bar{x}) \otimes (\bar{y}-1) & (T_{18}) \\
+(yx-1)\otimes (x-yx) \otimes (\bar{y}-x) \otimes (1-\bar{y})  & (T_{19})\\
-(yx-1)\otimes (1-yx) \otimes (\bar{y}-1) \otimes (1-\bar{y})  & (T_{20})\\
+(yx-1)\otimes (1-yx) \otimes ({y}-1) \otimes (1-{y})  & (T_{21})\\
+(y\bar{x}-1)\otimes (1-y\bar{x}) \otimes (\bar{y}-1) \otimes (1-\bar{y})  & (T_{22})\\
-(y\bar{x}-1)\otimes (1-y\bar{x}) \otimes({y}-1) \otimes (1-{y}) & (T_{23}) \\
+(y\bar{x}-1)\otimes (\bar{x}-y\bar{x}) \otimes({y}-\bar{x}) \otimes (1-{y}) & (T_{24}) \\
\end{array}\right.
}}
\]

\subsection{Proof of Proposition \ref{propKlein}} 
Recall that $\nu^ 4(\Omega_{\Z_2})$ is the image of $\nu^4(\Omega)$ in the coinvariants of $I(G;\Z_2)^{\otimes 4}$ with respect to the $G\times G$ action. Considering the image of $\nu^4(\Omega)$ in $I(G;\Z_2)^{\otimes 4}$ permits us to forget the signs in the expression above. 

We will carry out further projections which will enable us to see that this element is not zero. The major reduction comes from a projection onto the third exterior power $\bigwedge^3 I(D;\Z_2)$ where $D=\langle x,y\mid yxy=x,x^2=1\rangle$ is the infinite dihedral group.

Recall that $k$-th exterior power $\bigwedge^ k M$ ($k\geq 1$) of a $\Z$-module $M$ is given by $\bigwedge^ k M=M^{\otimes k}/R_k$ where $R_1=\{0\}$ and, for $k\geq 2$, $R_k$ is the submodule of $M^{\otimes k}$ spanned by all $m_1\otimes \cdots \otimes m_k$ with $m_i=m_j$ for some $i\neq j$. The class of  $m_1\otimes \cdots \otimes m_k$ in $\bigwedge^ k M$ is denoted by $m_1\wedge \cdots \wedge m_k$. If $M$ is a module over a group ring $\Z[\pi]$, then $\bigwedge^ k M$ is also a $\Z[\pi]$-module since $R_k$ is a sub-$\Z[\pi]$-module of $M^{\otimes k}$ under the diagonal action of $\pi$.

We also observe that, if $\rho\colon G \to H$ is a homomorphism, then we have a commutative diagram 
\[
\xymatrix{
\Z[G\times G]\otimes I(G;\Z_2) \ar[rr]^ {\ \Z[\rho\times \rho] \otimes I(\rho)\,\  } \ar[d] && \Z[H\times H]\otimes I(H;\Z_2) \ar[d]\\
I(G;\Z_2) \ar[rr]^ {I(\rho)}&& I(H;\Z_2)}
\]
which makes compatible the $\Z[G\times G]$-module and $\Z[H\times H]$-module structures.

\subsubsection{Projection onto $\bigwedge^3 I(D;\Z_2)$}

The projection 
\[
p\colon I(G;\Z_2)^{\otimes 4} \longrightarrow \hbox{\text{$\bigwedge^3$}} I(D;\Z_2)
\]
we will use is the composition of the following two projections:

\begin{itemize}
\item[(i)] $I(G;\Z_2)^{\otimes 4}= I(G;\Z_2) \otimes I(G;\Z_2)^{\otimes 3} \longrightarrow I(Y;\Z_2) \otimes I(D;\Z_2)^{\otimes 3}$\\[3pt]
Here $Y=\langle x,y|yxy=x,x=1\rangle=\langle y|y^2=1\rangle$ and the projections on the factors are induced by the homomorphisms $G\to Y$ and $G\to D$. Since $I(Y;\Z_2)\cong \Z_2(y-1)\cong \Z_2$, we suppress the first component in the continuation of the calculation.

\medskip

\item[(ii)] $I(Y;\Z_2)\otimes (I(D;\Z^2))^{\otimes 3}\cong (I(D;\Z^2))^{\otimes 3}\longrightarrow \hbox{\text{$\bigwedge^3$}}I(D;\Z_2)$
\end{itemize}

Through manipulations in $\bigwedge^3I(D;\Z_2)$, together with the fact that we are working with $\Z_2$ coefficients, we can see that 

\begin{itemize}
\item $p(T_i)=0$ for all $T_i$ except for $T_1$, $T_6$, $T_{10}$, $T_{15}$, $T_{19}$ and $T_{24}$;
\smallskip
\item $p(T_1+T_{10})=p(T_6+T_{15})=0$;
\smallskip
\item $p(T_{19}+T_{24})=(x-1)\wedge (yx-1)\wedge (y-\bar{y})$. 
\end{itemize}
We thus obtain that the image of $\nu^4(\Omega)$ under the projection $p$ is the element
\[
s=(x-1)\wedge (yx-1)\wedge (y-\bar{y})\in \hbox{\text{$\bigwedge^3$}} I(D;\Z_2).
\]
As a consequence, $\nu^4(\Omega_{\Z_2})$ is not trivial if the class of $s$ is not trivial in the coinvariants $(\bigwedge^3 I(D;\Z_2))_{G\times G}$.
We note that since the homomorphism $G\times G\to D\times D$ is surjective and its kernel $\langle x^2\rangle\times \langle x^2\rangle$ acts trivially on $\hbox{\text{$\bigwedge^3$}} I(D;\Z_2)$, we actually have 
\[
(\hbox{\text{$\bigwedge^3$}} I(D;\Z_2))_{G\times G}=(\hbox{\text{$\bigwedge^3$}} I(D;\Z_2))_{D\times D}
\]
and it suffices to show that $s$ is not trivial modulo the $D\times D$ action.

\begin{rem} Above, we have attempted to present this projection as efficiently as possible. However, it might be worth noting that it was developed in another order. More precisely, after noting that the image of $\nu^4(\Omega_{\widetilde{\Z}})$ under the projection  
\[
I(G)^{\otimes 4}\otimes_{G\times G}\widetilde{\Z} \longrightarrow 
\hbox{\text{$\bigwedge^4$}}I(G)\otimes_{G\times G}\widetilde{\Z}
\]
was trivial, we noticed that its image in $(I(G)\otimes \bigwedge^3I(G))\otimes_{G\times G}\widetilde{\Z}$ was apparently nontrivial, and that we were not losing this information by further reducing to the infinite dihedral group and $\Z_2$ coefficients.
\end{rem}

\subsubsection{Class of $s$ modulo the $D\times D$ action}
We first note that 
\[
D=\{y^n,y^nx\mid n\in \Z\},
\] 
and that, consequently, a basis of the $\Z_2$ vector space $I(D;\Z_2)$ is given by the elements
\[
y^m-1, \,\,y^nx-1, \quad m\in \Z\setminus\{0\}, n\in \Z.
\]
Let $w=y-1$, $v=x-1$ and let $(w^ 3)$ be the sub-$\Z[D\times D]$-module of $I(D;\Z_2)$ (and of $\Z_2[D]$) generated by $w^3$. The quotient $J=I(D;\Z_2)/(w^ 3)$ is a sub-$\Z[D\times D]$-module of  $\Z_2[D]/(w^ 3)$ which can be described as
the $\Z_2$ vector space with basis
\[
v,w,wx,w^2,w^2x
\]
and $D\times D$ action given by the following table:
\[
\arraycolsep=1.4pt\def\arraystretch{1.2}
\begin{array}{c||c|c|c|c|c|}
& v&w&wx&w^2&w^2x\\
\hline
\hline
(x,1) & v&\,wx+w^2x\,&w+w^2&\, w^2x\, &w^2\\
\hline
(1,x) & v&wx&w&w^2x&w^2\\
\hline
(y,1) & v+w+wx&w+w^2&wx+w^2x&w^2&w^2x\\
\hline
(1,y) & v+w+w^2+wx&w+w^2&wx+w^2x&w^2&w^2x\\
\hline
(\bar{y},1) & \, v+w+w^2+wx+w^2x\, &\, w+w^2\,&\, wx+w^2x\, &\, w^2\, &\, w^2x\, \\
\hline
(1,\bar{y}) & v+w+wx+w^2x&w+w^2&wx+w^2x&w^2&w^2x\\
\hline
\end{array}
\]
This description follows from the fact that $y^nx-1=(y^n-1)x+(x-1)$ and that, in $J$, we have
\[
y^n-1=\left\{
\begin{array}{cl}
0, &\quad \mbox{if } n\equiv 0 \,\,{\rm mod} \,4,\\
w, &\quad \mbox{if } n\equiv 1 \,\,{\rm mod} \,4,\\
w^2, &\quad  \mbox{if } n\equiv 2 \,\,{\rm mod} \,4,\\
w+w^2, &\quad  \mbox{if } n\equiv 3 \,\,{\rm mod} \,4.\\
\end{array}
\right.
\]
Using the projection $\bigwedge^3 I(D;\Z_2)\to \bigwedge^3J$, we look at our element $s$ in $\bigwedge^3J$. We obtain
\[
\arraycolsep=1.4pt\def\arraystretch{1.2}
s=v\wedge wx\wedge w^2.
\]

Let $L={\rm span}_{\Z_2}\{w,wx,w^2,w^2x\}$. This is a sub-$\Z[D\times D]$-module of $J$ and we can consider the projection
\[
\hbox{\text{$\bigwedge^3$}}J\longrightarrow \hbox{\text{$\bigwedge^3$}} J/\hbox{\text{$\bigwedge^3$}}L.
\]
The quotient $\bigwedge^3 J/\bigwedge^3L$ is the $\Z_2$ vector space 
with basis $\{a,b,c,d,e,f\}$, where
\[
\begin{matrix}
a=v\wedge w \wedge wx,\qquad & b=v\wedge w \wedge w^2,\qquad & c=v\wedge w \wedge w^2x,\\[2pt]
d=v\wedge wx \wedge w^2,\hfill& e=v\wedge wx \wedge w^2x,\hfill & f=v\wedge w^2 \wedge w^2x.\hfill
\end{matrix}
\]
The computation of the $D\times D$ action gives:
\[
\arraycolsep=1.4pt\def\arraystretch{1.2}
\begin{array}{c||c|c|c|c|c|c|}
& a&b&c&d&e&f\\
\hline
\hline
(x,1) & \,a+c+d+f\,&\,e\,&\,d+f\,&\,c+f\,&\,b\,&\,f\,\\
\hline
(1,x) & a&e&d&c&b&f\\
\hline
(y,1) & a+c+d+f&b&c+f&d+f&e&f\\
\hline
\end{array}
\]
The action of each of the elements $(1,y)$, $(\bar{y},1)$, and $(1,\bar{y})$ is the same as that of $(y,1)$.

We can thus see that, modulo the $D\times D$ action, we have 
\[
b=e,\qquad c=d, \qquad f=0,
\]
and that there are no more relations. Therefore our element $s$, which corresponds to $d=c$, is not $0$ modulo the $D\times D$ action. This completes the proof of Proposition~\ref{propKlein}.

\section{Topological complexity of $N_g$, $g\geq 2$}\label{SectionNg}

We use the following presentation of the fundamental group of $N_g$:
\[
\pi_1(N_g)=\langle a_1,\dots,a_g \,|\, a_1^ 2\cdots a_g^ 2=1\rangle,
\]
and consider, for $g\geq 3$, the map $\phi\colon N_g\to N_{g-1}$ induced (up to homotopy) by the homomorphism $\varphi\colon \pi_1(N_g)\to \pi_1(N_{g-1})$
\[
a_i \mapsto  a_i, \quad \mbox{for } 1\leq i\leq g-1 \quad \mbox{and } a_g\mapsto 1.
\]

\begin{lem} \label{lem:lem11} For $g\geq 3$, 
$H_2(\phi)\colon H_2(N_g; \Z_2) \to H_2(N_{g-1};\Z_2)$ is an isomorphism.
\end{lem}

\begin{proof} Again, a finite free $\Z[\pi_1(N_g)]$-resolution of $\Z$ can be obtained using Fox calculus. Explicitly, writing $\alpha=a_1^2\cdots a_g^2$, a finite free $\Z[\pi_1(N_g)]$-resolution of $\Z$ is given by (see, for instance, \cite{MartinsGoncalves}):
\[
\xymatrix{
M^g_2\ar[r]^{\partial} & M^g_1 \ar[r]^{\partial} & M^g_0 \ar[r]^{\varepsilon}&\Z
}
\]
where $M^g_0$, $M^g_1$, and $M^g_2$ are free $\Z[\pi_1(N_g)]$-modules with bases $\{e^0\}$, $\{f^1_1,\dots,f^1_g\}$, and $\{\omega_g\}$ respectively. The differential is given by
\[
\partial(f^1_i)=(a_i-1)e^0, \quad \partial(\omega_g)=\sum\limits_{i=1}^g\frac{\partial \alpha}{\partial a_i}f^1_i=\sum\limits_{i=1}^g a_1^2\cdots a^2_{i-1}(1+a_i)f^1_i.
\]
The differential $\partial(\omega_g)$ is given above in terms of the Fox free derivatives of $\alpha$.
The map
\[
e^0\mapsto e^0, \quad 
\left\{\begin{array}{l}
f^1_i\mapsto f^1_i\,\, {(1\leq i\leq g-1),} \\[2pt]
f^1_g \mapsto 0,
\end{array}
\right.
\quad
\omega_g \mapsto \omega_{g-1}
\]
gives a chain map $M^g_{\bullet} \to M^{g-1}_{\bullet}$ which induces the required isomorphism. 
\end{proof}

\begin{thm} For $g\geq 2$, $\TC(N_g)=4$.
\end{thm}
\begin{proof} Let $g\geq 3$.
Writing $\pi_g$ and $\pi_{g-1}$ in place of $\pi_1(N_g)$ and $\pi_1(N_{g-1})$ and $\nu_g$, $\nu_{g-1}$ for the associated canonical cocycles, the homomorphism $\varphi$ induces a commutative diagram
$$\xymatrix{
B_{4}(\pi_g\times \pi_g)\otimes_{\pi_g\times \pi_g}{\Z_2} \ar[d]\ar[rr]^{\nu_g^4\otimes\,\mathrm{id}} && I(\pi_g)^{\otimes 4}\otimes_{\pi_g\times \pi_g}{\Z_2}\ar[d]\\
B_{4}(\pi_{g-1}\times \pi_{g-1})\otimes_{\pi_{g-1}\times \pi_{g-1}}{\Z_2} \ar[rr]^{\nu_{g-1}^4\otimes\,\mathrm{id}} && I(\pi_{g-1})^{\otimes 4}\otimes_{\pi_{g-1}\times \pi_{g-1}}{\Z_2}.\\
}$$
By Lemma \ref{lem:lem11}, the map $H_4(\phi\times \phi)\colon H_4(N_g\times N_g;\Z_2)\to H_4(N_{g-1}\times N_{g-1};\Z_2)$ is an isomorphism, so that the left hand vertical map in the diagram above maps a cycle $\Omega^g_{\Z_2}$ representing the generator of $H_4(N_g\times N_g;\Z_2)$ to 
a cycle $\Omega^{g-1}_{\Z_2}$ representing the generator of $H_4(N_{g-1}\times N_{g-1};\Z_2)$. The commutativity of the diagram implies that 
$\nu^4_g(\Omega^g_{\Z_2})\neq 0$ as soon as $\nu^4_{g-1}(\Omega^{g-1}_{\Z_2})\neq 0$. Since $\nu^4_g(\Omega^g_{\Z_2})\neq 0$ for $g=2$, we therefore have $\nu^4_g(\Omega^g_{\Z_2})\neq 0$ for any $g\geq 2$. Proposition \ref{Z2coeff} permits us to complete the proof.
\end{proof}

\section*{Acknowledgements}
We thank Michael Farber for a number of valuable discussions concerning the theme of this article, Jes\'us Gonz\'alez and Mark Grant for many helpful suggestions on our first version, the anonymous referees for useful comments, and the Mathematisches Forschungsinstitut Oberwolfach, where this collaboration started during the mini-workshop \textit{Topological Complexity and Related Topics}.

\newcommand{\arxiv}[1]{{\texttt{\href{http://arxiv.org/abs/#1}{{arXiv:#1}}}}}

\newcommand{\MRh}[1]{\href{http://www.ams.org/mathscinet-getitem?mr=#1}{MR#1}}

\end{document}